\documentclass[preprint,12pt]{elsarticle}

\biboptions{numbers,sort&compress}
\usepackage{amssymb}
\usepackage{amsthm}
\usepackage{amsmath}
\usepackage{epic}
\usepackage{CJK}
\usepackage{setspace}%%行间距
\usepackage{bm}
\newtheorem{thm}{Theorem}[section]

\newtheorem{pro}[thm]{Proposition}

\newtheorem{example}[thm]{Example}

\journal{}

\begin{document}
\begin{spacing}{1.15}
\begin{CJK*}{GBK}{song}
\begin{frontmatter}
\title{\textbf{A combinatorial problem related to the classical probability}}
\author{Jiang Zhou}\ead{zhoujiang@hrbeu.edu.cn}
\address{College of Mathematical Sciences, Harbin Engineering University, Harbin 150001, PR China}

\begin{abstract}
In the classical probability model, let $f(n)$ be the maximum number of pairwise independent events for the sample space with $n$ sample points. The determination of $f(n)$ is equivalent to the problem of determining the maximum cardinality of specific intersecting families on the set $\{1,2,\ldots,n\}$ . We show that $f(n)\leq n+1$, and $f(n)=n+1$ if there exists a Hadamard matrix of order $n$.
\end{abstract}

\end{frontmatter}

\section{Introduction}
We only consider the classical probability model in this paper. Let $\Omega=\{1,2,\ldots,n\}$ denote the sample space with $n$ sample points. For an event $A\subseteq\Omega$, the probability of $A$ is
\begin{equation}
\mathbb{P}(A)=\frac{|A|}{|\Omega|}=\frac{|A|}{n}.\tag{1.1}
\end{equation}
Two events $A,B\subseteq\Omega$ are called \textit{independent} if $\mathbb{P}(A\cap B)=\mathbb{P}(A)\mathbb{P}(B)$. The $t$ distinct events $A_1,A_2,\ldots,A_t\subseteq\Omega$ are called \textit{pairwise independent} if $A_i$ and $A_j$ are independent for any $1\leq i<j\leq t$. By the formula (1.1), $A_1,A_2,\ldots,A_t\subseteq\Omega$ are pairwise independent if and only if
\begin{equation}
|A_i\cap A_j|=n^{-1}|A_i||A_j|\tag{1.2}
\end{equation}
for any $1\leq i<j\leq t$. Let $f(n)$ denote the maximum number of pairwise independent events for the sample space with $n$ sample points.

For a set $\Omega$, we say that $\mathcal{F}=\{A_1,A_2,\ldots,A_t\}$ is a \textit{set system} on $\Omega$ if $A_1,A_2,\ldots,A_t$ are distinct subsets of $\Omega$. A set system $\mathcal{F}$ is called an \textit{intersecting family} \cite{Bollobas} if $A\cap B\neq\emptyset$ for any $A,B\in\mathcal{F}$. Let $g(n)$ denote the maximum cardinality of intersecting families on the set $\{1,2,\ldots,n\}$ satisfying (1.2). Clearly, we have $f(n)=g(n)+1$, because $A_1,A_2,\ldots,A_t,\emptyset$ are pairwise independent events if the intersecting family $\mathcal{F}=\{A_1,A_2,\ldots,A_t\}$ satisfies (1.2).

A $2$-$(v,k,\lambda)$ design \cite{Beth,Brouwer} is a set of $v$ points, with a set of blocks (subsets of points) of size $k$ such that each pair of points occurs in exactly $\lambda$ blocks. A $2$-$(v,k,\lambda)$ design is called \textit{symmetric} if the number of blocks is $v$. It is known \cite{Beth} that any two distinct blocks in a symmetric $2$-$(v,k,\lambda)$ design intersect in precisely $\lambda$ points. A $2$-$(v,k,\lambda)$ design can be regarded as a set system $\mathcal{F}=\{A_1,A_2,\ldots,A_t\}$ on $A=\{1,2,\ldots,v\}$ such that $|A_1|=\cdots=|A_t|=k$ and each pair of elements in $A$ occurs in exactly $\lambda$ subsets among $A_1,A_2,\ldots,A_t$.

In this paper, we show that $f(n)=g(n)+1\leq n+1$, and $f(n)=n+1$ if there exists a Hadamard matrix of order $n$. We also obtain lower bounds on $g(n)$ by using $2$-designs and generalized Johnson graphs.

\section{Main results}
Let $I_n$ denote the identity matrix of order $n$. A square matrix $H$ of order $n$, with entries $1$ or $-1$, is called a \textit{Hadamard matrix} if $HH^\top=nI_n$, where $H^\top$ denotes the transpose of $H$.
\begin{thm}\label{thm3.1}
For any positive integer $n$, we have
\begin{eqnarray*}
g(n)\leq n,
\end{eqnarray*}
and the equality holds if there exists a Hadamard matrix of order $n$.
\end{thm}
\begin{proof}
Let $\mathcal{F}=\{A_1,A_2,\ldots,A_t\}$ ($t=g(n)$) be a maximum intersecting family on $\Omega=\{1,2,\ldots,n\}$ satisfying (1.2). If $\Omega\notin\mathcal{F}$, then $\{A_1,A_2,\ldots,A_t,\Omega\}$ is a larger intersecting family satisfying (1.2). So $\Omega\in\mathcal{F}$. Without loss of generality, assume that $A_t=\Omega$. Let $B$ be the $n\times t$ (incidence) matrix with entries
\begin{eqnarray*}
(B)_{ij}=\begin{cases}1~~~~~~~~~~~~~~~~~~\mbox{if}~i\in A_j,\\
0~~~~~~~~~~~~~~~~~~\mbox{if}~i\notin A_j.\end{cases}
\end{eqnarray*}
Let $u=(|A_1|,|A_2|,\ldots,|A_t|)^\top$. By (1.2), we have
\begin{eqnarray*}
B^\top B=n^{-1}uu^\top+D,
\end{eqnarray*}
where $D$ is the diagonal matrix with entries
\begin{eqnarray*}
(D)_{ii}=|A_i|-n^{-1}|A_i|^2=|A_i|(1-n^{-1}|A_i|)~(i=1,2,\ldots,t).
\end{eqnarray*}

Since $0<|A_i|<n$ for each $i\in\{1,2,\ldots,t-1\}$ and $|A_t|=n$, we have
\begin{eqnarray*}
(D)_{ii}&>&0~(i=1,2,\ldots,t-1),\\
(D)_{tt}&=&0.
\end{eqnarray*}
For any nonzero vector $x=(x_1,x_2,\ldots,x_t)^\top\in\mathbb{R}^t$, we have
\begin{eqnarray*}
x^\top B^\top Bx&=&n^{-1}x^\top uu^\top x+x^\top Dx\\
&=&n^{-1}\left(\sum_{i=1}^t|A_i|x_i\right)^2+\sum_{i=1}^t(D)_{ii}x_i^2>0.
\end{eqnarray*}
Hence $B^\top B$ is positive definite. Since $B$ is an $n\times t$ matrix, we have
\begin{eqnarray*}
g(n)=t=\mbox{\rm rank}(B^\top B)=\mbox{\rm rank}(B)\leq n.
\end{eqnarray*}

If there exists a Hadamard matrix of order $n$, then it is known \cite{Haemers} that there exists a symmetric $2$-$(n-1,n/2-1,n/4-1)$ design. Then there exist subsets $B_1,B_2,\ldots,B_{n-1}\subseteq\{1,2,\ldots,n-1\}$ such that
\begin{eqnarray*}
|B_i|&=&n/2-1~(i=1,2,\ldots,n-1),\\
|B_i\cap B_j|&=&n/4-1~(1\leq i<j\leq n-1).
\end{eqnarray*}
Set $C_i=B_i\cup\{n\}$ ($i=1,2,\ldots,n-1$). Then $\mathcal{F}_0=\{C_1,C_2,\ldots,C_{n-1},\Omega\}$ is a maximum intersecting family on $\Omega=\{1,2,\ldots,n\}$ satisfying (1.2). In this case, we have $g(n)=n$.
\end{proof}
Since $f(n)=g(n)+1$, we can derive the following equivalent theorem from Theorem \ref{thm3.1}.
\begin{thm}\label{thm2.2}
For any positive integer $n$, we have
\begin{eqnarray*}
f(n)\leq n+1,
\end{eqnarray*}
and the equality holds if there exists a Hadamard matrix of order $n$.
\end{thm}
We can derive the following lower bound of $g(n)$ by using $2$-designs.
\begin{thm}\label{thm2.3}
If there exists a $2$-$(v,k,\lambda)$ design such that $\frac{\lambda(v-1)^2}{(k-1)^2}=n$ for some positive integer $n$, then
\begin{eqnarray*}
g(n)\geq v+1.
\end{eqnarray*}
\end{thm}
\begin{proof}
Let $r=\frac{\lambda(v-1)}{(k-1)}$. Suppose that there exists a $2$-$(v,k,\lambda)$ design $(V,\mathcal{B})$ such that $r^2=\lambda n$ for some positive integer $n$, where $V=\{p_1,p_2,\ldots,p_v\}$ is the set of $v$ points, $\mathcal{B}=\{B_1,B_2,\ldots,B_b\}$ is the set of $b$ blocks. Define $A_i=\{j|p_i\in B_j\}$, $i=1,2,\ldots,v$. Then
\begin{eqnarray*}
|A_i|&=&r~(i=1,2,\ldots,v),\\
|A_i\cap A_j|&=&\lambda~(1\leq i<j\leq v).
\end{eqnarray*}
Since $vr=kb$ and $r(k-1)=\lambda(v-1)$, we have
\begin{eqnarray*}
\lambda n=r^2=\frac{b(r+(v-1)\lambda)}{v}\geq b\lambda.
\end{eqnarray*}
Hence $\mathcal{F}=\{A_1,\ldots,A_v,\Omega\}$ is an intersecting family on $\Omega=\{1,2,\ldots,n\}$ satisfying (1.2). Hence
\begin{eqnarray*}
g(n)\geq v+1.
\end{eqnarray*}
\end{proof}

\begin{example}
A symmetric $2$-$(q^2+q+1,q+1,1)$ design comes from geometries over a finite field $\mathbf{F}_q$, and such a design is called a projective plane of order $q$ (see \cite{Haemers}). By Theorems \ref{thm3.1} and \ref{thm2.3}, we have
\begin{eqnarray*}
q^2+q+2\leq g(q^2+2q+1)\leq q^2+2q+1
\end{eqnarray*}
if there exists a projective plane of order $q$.
\end{example}

\section{Remarks}
The \textit{clique number} of a graph $G$ is the maximum cardinality of cliques in $G$. For the set $\Omega=\{1,2,\ldots,n\}$, let $\mathcal{P}(\Omega)$ denote the power set of $\Omega$. The power set graph $G(\Omega)$ of $\Omega$ is defined as the graph with vertex set $\mathcal{P}(\Omega)$, and two vertices $A$ and $B$ are adjacent if and only if they satisfy the relation (1.2). It is easy to see that $f(n)$ equals to the clique number of $G(\Omega)$, and any maximum clique of $G(\Omega)$ must contain the vertices $\emptyset$ and $\Omega$.

The generalized Johnson graph \cite{Godsil} $J(n,r,s)$ is a graph whose vertices are all $r$-subsets of $\{1,2,\ldots,n\}$, and two vertices $A$ and $B$ are adjacent if and only if $|A\cap B|=s$. Generalized Johnson graphs are also called the uniform subset graphs \cite{Chen}. When $ns=r^2$ and $n>r>s$, $J(n,r,s)$ is an induced subgraph of the power set graph $G(\Omega)$. Hence we have the following lower bound of $f(n)$.
\begin{pro}\label{pro3.1}
Let $n>r>s$ be positive integers such that $ns=r^2$, and let $\omega(n,r,s)$ be the clique number of $J(n,r,s)$. Then
\begin{eqnarray*}
f(n)\geq\omega(n,r,s)+2.
\end{eqnarray*}
\end{pro}
If there exists a Hadamard matrix of order $n$, then from the proof of Theorem \ref{thm3.1}, there exist subsets $C_1,C_2,\ldots,C_{n-1}\subseteq\{1,2,\ldots,n\}$ such that
\begin{eqnarray*}
|C_i|&=&n/2~(i=1,2,\ldots,n-1),\\
|C_i\cap C_j|&=&n/4~(1\leq i<j\leq n-1).
\end{eqnarray*}
Then $\omega(n,n/2,n/4)\geq n-1$. By Proposition \ref{pro3.1} and Theorem \ref{thm2.2}, we know that $\omega(n,n/2,n/4)=n-1$ if there exists a Hadamard matrix of order $n$.

If $n$ is a prime number, then by (1.2), we have $g(n)=2$. It is conjectured that the $n\times n$ Hadamard matrix exists for all $n$ divisible by $4$. By observing Theorem \ref{thm3.1}, we conjecture that $g(n)=n$ when $n$ is divisible by $4$.

\end{CJK*}
\end{spacing}
\end{document}